\pdfoutput=1
\documentclass[11pt]{article}

\usepackage{amsmath,amsfonts,amssymb,latexsym,amsthm,dsfont}
\usepackage{geometry,graphicx}
\usepackage[backref]{hyperref}

\newcommand{\ind}{\mathds{1}}
\newtheorem{thm}{Theorem}[section]%
\newtheorem{cor}[thm]{Corollary}%
\newtheorem{prop}[thm]{Proposition}%
\newtheorem{lem}[thm]{Lemma}%

\theoremstyle{definition}
\newtheorem{defi}[thm]{Definition}%
\newtheorem{rem}[thm]{Remark}%

\newcommand{\fL}{\mathfrak{L}}

\newcommand{\dE}{\mathbb{E}}

\newcommand{\dP}{\mathbb{P}}
\newcommand{\dR}{\mathbb{R}}

\newcommand{\cC}{\mathcal{C}}
\newcommand{\cE}{\mathcal{E}}

\newcommand{\cL}{\mathcal{L}}

\newcommand{\cP}{\mathcal{P}}

\newcommand{\PENT}[1]{{\lfloor#1\rfloor}} 
\newcommand{\ABS}[1]{{{\left| #1 \right|}}} 
\newcommand{\BRA}[1]{{{\left\{#1\right\}}}} 
\newcommand{\NRM}[1]{{{\left\| #1\right\|}}} 
\newcommand{\TVNRM}[1]{\NRM{#1}_{\mathrm{TV}}} 
\newcommand{\PAR}[1]{{{\left(#1\right)}}} 
\newcommand{\SBRA}[1]{{{\left[#1\right]}}} 
\renewcommand{\leq}{\leqslant}
\renewcommand{\geq}{\geqslant}

\newcommand*{\abs}[1]{\left|#1\right|} 
\newcommand{\gen}{\mathfrak{L}}
\newcommand{\law}{\mathcal{L}}
\newcommand{\Vlyap}{\widetilde{V}}

\title{Total variation estimates for the TCP process}

\author{Jean-Baptiste~\textsc{Bardet}, %
Alejandra~\textsc{Christen},\\%
Arnaud~\textsc{Guillin}, %
Florent~\textsc{Malrieu}
and Pierre-Andr\'e~\textsc{Zitt}} %

\date{\today}

\begin{document}

\maketitle

\begin{abstract}
  The TCP window size process appears in the modeling of the famous
  Transmission Control Protocol used for data transmission over the Internet.
  This continuous time Markov process takes its values in $[0,\infty)$, is
  ergodic and irreversible. The sample paths are piecewise linear
  deterministic and the whole randomness of the dynamics comes from 
  the jump mechanism. The aim of the present paper is to provide 
  quantitative estimates for the exponential convergence to equilibrium, 
  in terms of the total variation and Wasserstein distances. 
\end{abstract}

{\footnotesize %

\medskip

\noindent\textbf{AMS-MSC.}  68M12 ; 60K30 ; 60K25 ; 90B18

 \tableofcontents
}

\section{Introduction and main results}

The TCP protocol is one of the main data transmission protocols of the
Internet. It has been designed to adapt to the various traffic conditions of
the actual network. For a connection, the maximum number of packets that can
be sent at each round is given by a variable $W$, called the \emph{congestion
  window size}. If all the $W$ packets are successfully transmitted, then $W$
is increased by $1$, otherwise it is divided by $2$ (detection
of a congestion). As shown in \cite{DGR,GRZ,ott-kemperman-mathis}, a correct
scaling of this process leads to a continuous time Markov process, called
the TCP window size process. This process $X={(X_t)}_{t\geq0}$ has
$[0,\infty)$ as state space and its infinitesimal generator is given, for any
smooth function $f:[0,\infty)\to\dR$, by
\begin{equation}\label{eq:G1}
  Lf(x) = f'(x) + x(f(x/2)-f(x)).
\end{equation}
The semi-group ${(P_t)}_{t\geq 0}$ associated to ${(X_t)}_{t\geq 0}$ is 
classically defined by 
\[
P_tf(x)=\dE\PAR{f(X_t)\vert X_0=x},
\]
for any smooth function $f$. Moreover, for any probability measure $\nu$ on 
$[0,\infty)$, $\nu P_t$ stands for the law of $X_t$ when $X_0$ is distributed 
according to $\nu$. 

The process $X_t$ increases linearly between jump times that occur with a 
state-dependent rate $X_t$. The first jump time of $X$ starting at $x\geq 0$ 
has the law of $\sqrt{2E+x^2}-x$ where $E$ is a random variable with 
exponential law of parameter 1. In other words, the law of this jump time 
has a density with respect to the Lebesgue measure on $(0,+\infty)$ given by 
\begin{equation}\label{eq:def-fx}
 f_x\,:\, t\in(0,+\infty)\mapsto (x+t) e^{-t^2/2-xt},
\end{equation}
 see \cite{CMP10} for further details. 

The sample paths of $X$ are deterministic
between jumps, the jumps are multiplicative, and the whole randomness of the
dynamics relies on the jump mechanism. Of course, the randomness of $X$ may
also come from a random initial value. The process ${(X_t)}_{t\geq0}$ appears
as an Additive Increase Multiplicative Decrease process (AIMD), but also as a
very special Piecewise Deterministic Markov Process (PDMP). In this direction,
\cite{maulik-zwart} gives a generalization of the scaling procedure to
interpret various PDMPs as limits of discrete time Markov chains. In the
real world (Internet), the AIMD mechanism allows a good compromise 
between the minimization of network congestion time and the maximization 
of mean throughput. One could consider more general processes 
(introducing a random multiplicative factor or modifying the jump rate) 
but their study is essentially the same than the one of the present process. 

The TCP window size process $X$ is ergodic and admits a unique 
invariant law $\mu$, as can be checked using a suitable
Lyapunov function (for instance $V(x)=1+x$, see e.g.
\cite{MR2381160,MR2385873,MR1287609,HM} for the 
Meyn-Tweedie-Foster-Lyapunov technique). It can also be shown that 
$\mu$ has a density on $(0,+\infty)$ given by 
\begin{equation}\label{eq:invariant}
x\in(0,+\infty) \mapsto 
\frac{\sqrt{2/\pi}}{\prod_{n\geq 0}\PAR{1-2^{-(2n+1)}}}
\sum_{n\geq 0}\frac{(-1)^n2^{2n}}{\prod_{k=1}^n\PAR{2^{2k}-1}} 
e^{-2^{2n-1}x^2},
\end{equation}
(this explicit formula is derived in \cite[Prop. 10]{DGR}, see also 
\cite{GRZ,maulik-zwart,MR2197972,GK} for further details). In 
particular, one can notice that the density of $\mu$ has a Gaussian 
tail at $+\infty$ and that all its derivatives are null at the origin. 
Nevertheless, this process is irreversible since time reversed
sample paths are not sample paths and it has infinite support 
(see \cite{LP} for the description of the reversed process). 
In \cite{RR96}, explicit bounds for the exponential rate of convergence to 
equilibrium in total variation distance are provided for generic Markov 
processes in terms of a suitable Lyapunov function but these estimates are 
very poor even for classical examples as the Ornstein-Uhlenbeck 
process. They can be improved following~\cite{RT} if the process 
under study is \emph{stochastically monotone} that is if its semi-group 
${(P_t)}_{t\geq 0}$ is such that $x\mapsto P_tf(x)$ is nondecreasing 
as soon as $f$ is nondecreasing. Unfortunately, due to the fact that 
the jump rate is an nondecreasing function of the position, the TCP process 
is not stochastically monotone. Moreover, we will see that our coupling provides 
better estimates for the example studied in~\cite{RT}. 

The work \cite{CMP10} was a first attempt to use the specific dynamics of the 
TCP process to get explicit rates of convergence of the law of $X_t$ to 
the invariant measure $\mu$. The answer was partial and a bit 
disappointing since the authors did not succeed in proving 
explicit exponential rates. 

Our aim in the present paper is to go one step further providing 
exponential rate of convergence for several classical distances: 
Wasserstein distance of order $p\geq 1$ and total variation distance. 
Let us recall briefly some definitions. 

\begin{defi}
 If $\nu$ and $\tilde \nu$ are two probability measures on $\dR$, we 
 will call a \emph{coupling} of $\nu$ and $\tilde \nu$ any probability measure 
 on $\dR\times\dR$ such that the two marginals are $\nu$ and $\tilde \nu$. 
 Let us denote by $\Gamma(\nu,\tilde \nu)$ the set of all the couplings of $\nu$ 
 and $\tilde \nu$. 
\end{defi}

\begin{defi}
For every $p\geq1$, the \emph{Wasserstein distance} $W_p$  of order $p$ 
between two probability measures $\nu$ and $\tilde \nu$ on $\dR$
with finite $p^\text{th}$ moment is defined by
\begin{equation}\label{eq:Wp}
  W_p(\nu,\tilde \nu) %
  = \left(\inf_{\Pi\in \Gamma(\nu,\tilde \nu)}
  \int_{\dR^2}\!|x-y|^p\,\Pi(dx,dy)\right)^{1/p}.
\end{equation}
\end{defi}

\begin{defi}
The \emph{total variation distance} between two probability measures $\nu$ and 
$\tilde\nu$ on $\dR$ is given by 
\[
\NRM{\nu-\tilde \nu}_{\mathrm{TV}}=
\inf_{\Pi\in \Gamma(\nu,\tilde \nu)} \int_{\dR^2}\ind_\BRA{x\neq y}\,\Pi(dx,dy).
\] 
\end{defi}
 
It is well known that, for any $p\geq1$, the convergence in 
Wasserstein distance of order $p$ is equivalent to weak 
convergence together with convergence of all moments up 
to order~$p$, see e.g. \cite{MR1105086,MR1964483}. 
A sequence of probability measures ${(\nu_n)}_{n\geq 1}$ 
bounded in $L^p$ which converges to $\nu$ in total variation norm 
converges also for the $W_p$ metrics. The converse is false: 
if $\nu_n=\delta_{1/n}$ then ${(\nu_n)}_{n\geq 1}$ converges to 
$\delta_0$ for the distance $W_p$ whereas 
$\NRM{\nu_n-\delta_0}_\mathrm{TV}$ is equal to 1 for any $n\geq 1$. 

Any coupling $(X,\tilde X)$ of $(\nu,\tilde \nu)$ provides an upper 
bound for these distances. One can find in \cite{lindvall} a lot of efficient 
ways to construct smart couplings in many cases. In the present 
work, we essentially use the coupling that was introduced 
in \cite{CMP10}. Firstly, we improve the estimate for its rate of convergence 
in Wasserstein distances from polynomial to exponential bounds.

\begin{thm}\label{th:wasserstein} 
Let us define 
\begin{equation}\label{eq:def-M-lambda}
M=\frac{\sqrt{2}(3+\sqrt{3})}8 \sim 0.84 
\quad\text{and}\quad
\lambda=\sqrt 2(1-\sqrt M)\sim 0.12.  
\end{equation}
For any $\tilde \lambda<\lambda$, any  $p\geq1$ and any $t_0>0$, 
there is a constant $C=C(p,\tilde \lambda,t_0)$ such that,  
for any initial probability measures $\nu$ and $\tilde \nu$ and 
any $t\geq t_0$,   
\[
W_p(\nu P_t,\tilde \nu P_t)\leq C \exp\PAR{-\frac{\tilde \lambda}pt}.
\] 
\end{thm}

 Secondly, we introduce a modified coupling to get total variation estimates.  
\begin{thm}\label{th:variation}
For any $\tilde \lambda<\lambda$ and any $t_0>0$, there exists $C$ such that, 
for any initial probability measures $\nu$ and $\tilde \nu$ and 
any $t\geq t_0$, 
\[
\TVNRM{ \nu P_t-\tilde \nu P_t}\leq 
C\exp\PAR{-\frac{2\tilde \lambda}{3}t},
\]
where $\lambda$ is given by \eqref{eq:def-M-lambda}.
 \end{thm}

\begin{rem}
 In both Theorems~\ref{th:wasserstein} and \ref{th:variation}, no assumption 
 is required on the moments nor regularity of the initial measures. Note however that following Remark \ref{expcont}, one can obtain contraction's type bounds when the initial measures $\nu$ and $\tilde\nu$ have initial moments of sufficient orders. In particular 
 they hold uniformly over the Dirac measures. If $\tilde \nu$ is chosen to be 
 the invariant measure $\mu$, these theorems provide exponential 
 convergence to equilibrium. 
\end{rem}

The remainder of the paper is organized as follows. 
We derive in Section~\ref{se:moments} precise upper bounds for the 
moments of the invariant measure $\mu$ and the law of $X_t$. 
Section~\ref{se:wasserstein} and Section~\ref{se:variation} are 
respectively devoted to the proofs of Theorem~\ref{th:wasserstein} 
and Theorem~\ref{th:variation}. Unlike the classical approach "\`a la" 
Meyn-Tweedie, our total variation estimate is obtained by applying 
a Wasserstein coupling for most of the time, then trying to couple the 
two paths in one attempt. This idea is then adapted to others processes: 
 Section~\ref{sec:examples} deals with two simpler PDMPs already 
 studied in \cite{PR,Lau-Per,CMP10,RT} and Section~\ref{sec:diffusion} 
 is dedicated to diffusion processes.

\section{Moment estimates}\label{se:moments}

The aim of this section is to provide accurate bounds for the moments 
of $X_t$. In particular, we establish below that any moment of 
$X_t$ is bounded uniformly over the initial value $X_0$. 
Let $p>0$ and $\alpha_{p}(t)=\dE(X_{t}^p)$. Then one has by direct computation
\begin{equation}
\label{eq:edmom1}
\alpha_{p}'(t)=p\alpha_{p-1}(t)-\left(1-2^{-p}\right)\alpha_{p+1}(t).
\end{equation}

\subsection{Moments of the invariant measure}

Equation~\eqref{eq:edmom1} implies in particular that, if $m_{p}$ denotes the 
$p$-th moment of the invariant measure $\mu$ of the TCP process 
($m_{p}=\int x^p \mu(dx)$), then for any $p>0$
\[
m_{p+1}=\frac p {1-2^{-p}}m_{p-1}.
\]
It gives all even moments of $\mu$: $m_{2}=2$, $m_{4}=\frac{48}7$, 
$\ldots$ and all the odd moments in terms of the mean. Nevertheless, 
the mean itself cannot be explicitly determined. Applying the 
same technique to $\log X_t$, one gets the relation
$\log(2)m_1=m_{-1}$.
With Jensen's inequality, this implies that 
$1/{\sqrt{\log 2}} \leq m_{1} \leq \sqrt{2}$.

 \subsection{Uniform bounds for the moments at finite times}
The fact that the jump rate goes to infinity at infinity gives bounds on the 
moments at any positive time that are uniform over the initial distribution. 
\begin{lem}
\label{le:momest}
 For any $p\geq 1$ and $t>0$ 
\[
 M_{p,t}:=\sup_{x\geq 0}\dE_x\PAR{X_t^p}\leq \PAR{\sqrt{2p}+\frac{2p}{t}}^p. 
\]
\end{lem}

\begin{proof}  
One deduces from \eqref{eq:edmom1} and Jensen's inequality that 
\[
\alpha'_{p}(t)\leq p\alpha_{p}(t)^{1-1/p}-(1-2^{-p})\alpha_{p}(t)^{1+1/p}.
\]
Let $\beta_{p}(t)=\alpha_{p}(t)^{1/p}-\sqrt{2p}$. Then, using the fact
that $1-2^{-p}\geq\frac12$,
\[
 \beta_{p}'(t)
 =\frac1p\alpha_{p}(t)^{1/p-1}\alpha'_{p}(t) \leq 1-\frac1{2p}\alpha_{p}(t)^{2/p} 
 =-\sqrt{\frac2p}\beta_{p}(t)-\frac{\beta_{p}(t)^2}{2p}.
\]
In particular, $\beta_{p}'(t)<0$ as soon as $\beta_{p}(t)>0$. 

Let us now fix $t>0$. If $\beta_{p}(t)\leq0$, then $\alpha_{p}(t)\leq {(2p)}^{p/2}$ 
and the lemma is proven. We assume now that $\beta_{p}(t)>0$. By the  previous
remark, this implies that the function $s\mapsto \beta_{p}(s)$ is 
strictly decreasing, hence positive, on the interval $[0,t]$. 
Consequently, for any $s\in[0,t]$,
\[
\beta_{p}'(s)\leq -\frac{\beta_{p}(s)^2}{2p}.
\]
Integrating this last inequality gives 
\[
\frac1{\beta_{p}(t)}\geq\frac1{\beta_{p}(0)}+\frac t{2p}\geq \frac t{2p},
\] 
hence the lemma.
\end{proof}

Let us derive from Lemma~\ref{le:momest} some upper bounds for the 
right tails of $\delta_xP_t$ and $\mu$. 

\begin{cor}
 For any $t>0$ and $r\geq 2e(1+1/t)$, one has 
\begin{equation}\label{eq:deviation-t}
 \dP_x\PAR{X_t\geq r}\leq  \exp\PAR{-\frac{t}{2e(t+1)}r }.
\end{equation} 
Moreover, if $X$ is distributed according to the invariant measure $\mu$ 
then, for any $r\geq \sqrt{2e}$, one has 
\[
\dP(X\geq r)\leq \exp\PAR{-\frac{r^2}{4e}}.
\]
\end{cor}

\begin{proof}
Let $t>0$ and $a=2(1+1/t)$. Notice that, for any $p\geq 1$, $\dE_x(X_t^p)$ 
is smaller than $(ap)^p$. As a consequence, for any $p\geq 1$ and $r\geq 0$, 
\[
\dP_x(X_t\geq r) \leq \exp\PAR{p\log(ap/r)}.
\] 
Assuming that $r\geq ea$, we let $p=r/(ea)$ to get:
\[
\dP_x(X_t\geq r) \leq e^{-r/(ea)}.
\] 
For the invariant measure, the upper bound is better: $\dE(X^p)\leq (2p)^{p/2}$. 
Then, the Markov inequality provides that, for any $p \geq 1$, 
\[
\dP(X\geq r)\leq \exp\PAR{p\log\frac{\sqrt{2p}}{r}}.
\]
As above, if $r^2\geq 2e$, one can choose $p=r^2/(2e)$ to get the desired 
bound. 
\end{proof}

\begin{rem}
 A better deviation bound should be expected from the 
 expression~\eqref{eq:invariant} of the density of $\mu$. Indeed, one can 
 get a sharp result (see \cite{CMP10}). However, in the sequel we only 
 need the deviation bound~\eqref{eq:deviation-t}. 
\end{rem}

\section{Exponential convergence in Wasserstein distance}\label{se:wasserstein}

This section is devoted to the proof of Theorem~\ref{th:wasserstein}. 
We use the coupling introduced in~\cite{CMP10}. Let us briefly recall 
the construction and the dynamics of this stochastic process 
on $\dR_+^2$ whose marginals are two TCP processes. It is defined 
by the following generator
\[
\fL f(x,y)
=(\partial_{x}+\partial_{y})f(x,y)+y\big(f(x/2,y/2)-f(x,y)\big)
+(x-y)\big(f(x/2,y)-f(x,y)\big)
\]
when $x\geq y$ and symmetric expression for $x<y$. We will call the 
dynamical coupling defined by this generator the Wasserstein coupling 
of the TCP process (see Figure \ref{fig:traj_wass} for a graphical illustration 
of this coupling). This coupling is the only one such that the lower 
component never jumps alone. Let us give the pathwise interpretation 
of this coupling. Between two jump times, the two coordinates increase 
linearly with rate 1. Moreover, two "jump processes" are simultaneously 
in action:
\begin{enumerate}
\item with a rate equal to the minimum of the two coordinates, they 
jump (\emph{i.e.} they are divided by 2) simultaneously, 
\item with a rate equal to the distance between the two coordinates 
(which is constant between two jump times), the bigger one jumps 
alone.

\end{enumerate}
\begin{figure}[htbp]
  \centering
  \includegraphics[width=0.55\linewidth]{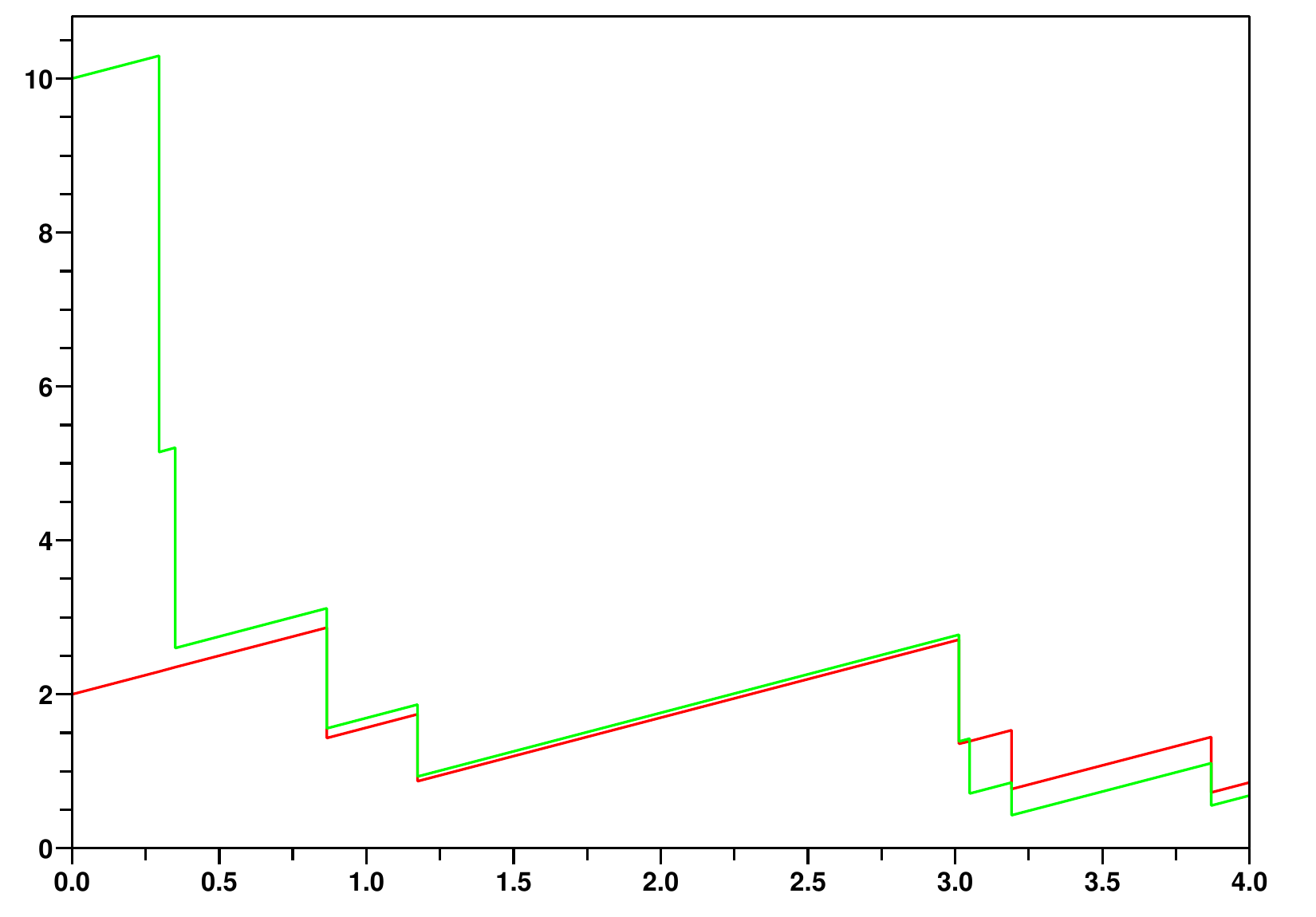}
  \caption{Two trajectories following the Wasserstein coupling; the bigger 
  jumping alone can be good, making the distance between both 
  trajectories smaller, or bad.}
  \label{fig:traj_wass}
\end{figure}
 The existence of non-simultaneous jumps implies 
 that the generator does not act in a good way on functions 
 directly associated to Wasserstein distances.
 To see this, let us define $V_{p}(x,y) = \abs{x -y}^p$. 
When we compute $\gen V_{p}$, the term coming from the deterministic 
drift disappears (since $V_{p}$ is unchanged under the flow), and we get for 
$x/2\leq y\leq x$:
\begin{align*}
  \gen V_p(x,y) = - y( 1 - {2^{-p}}) V_{p}(x,y) + (x-y) ( (y - x/2)^p - (x-y)^p ).
\end{align*}
For example, choosing $p=1$ gives:
\[ 
\gen V_1(x,y) = - V_1(x,y) ( y/2 - (y-x/2) + (x-y))  = - (3/2) V_1(x,y)^2.
\]
This shows that $\dE[\abs{X_t - Y_t}]$ decreases, but only gives a 
polynomial bound: the problem comes from 
the region where $x-y$ is already very small.

Choosing $p=2$, we get
\begin{align*}
  \gen V_2(x,y) 
  &= - (3/4)yV_2(x,y) + (x-y) ( -(3/4) x^2 + xy ).
\end{align*}
The effect of the jumps of $X$ is even worse: if $x = 1$ and  $x-y$ is small, 
 $\gen V_2(x,y)$ is positive ($\approx (1/4)(x-y)$). 

\medskip
For $p=1/2$, the situation is in fact more favorable: indeed, for $0< y\leq x$,
\begin{align*}
  \gen V_{1/2}(x,y) 
  &= - y \Big(1 - \frac{\sqrt{2}}{2}\Big) V_{1/2}(x,y) 
  + (x-y) ( \sqrt{|y - x/2|} - \sqrt{x-y} ) \\
  &= -V_{1/2}(x,y) \left[ x-\frac{\sqrt{2}}{2}y - \sqrt{(x-y)(y-x/2)} \right]\\
  &= -x\Big[1-\varphi\Big(\frac yx\Big)\Big]V_{1/2}(x,y)\,,  
\end{align*}
with 
\[
\varphi(u)=\frac{\sqrt{2}}2u+\sqrt{(1-u)|u-1/2|}
\quad\text{for}\quad u\in[0,1].
\]
 By a direct computation, one gets that 
\[
M:=\max_{0\leq u\leq 1}\varphi(u)=
\varphi\Big(\frac{9+\sqrt{3}}{12}\Big)=\frac{\sqrt{2}(3+\sqrt{3})}8
\sim 0.8365.
\] 
Hence, when $0< y\leq x$, 
\begin{align}
\label{eq:V12}  \gen V_{1/2}(x,y) \leq - x \lambda V_{1/2}(x,y),
\end{align}
with $\lambda=1-M\sim 0.1635$. This would give an exponential decay 
for $V_{1/2}$ if $x$ was bounded below: the problem comes from the 
region where $x$ is small and the process does not jump. 

\medskip

To overcome this problem, we replace $V_{1/2}$ with the function
\[
\Vlyap(x,y)=\psi(x\vee y)|x-y|^{1/2}=\psi(x\vee y)V_{1/2}(x,y),
\]
where 
\[
\psi(x)=\begin{cases} 1 + \alpha(1-x/x_{0})^2, & x\leq x_{0},\\
  1, & x\geq x_{0},
\end{cases}
\]
the two positive parameters $\alpha$ and $x_0$ being chosen below. 
The negative slope of $\psi$ for $x$ small will be able to compensate 
the fact that the bound from \eqref{eq:V12} tends to $0$ with $x$, hence 
to give a uniform bound. Indeed, for $0< y\leq x$,
\begin{align*}
  \gen \Vlyap (x,y) 
  & = \psi'(x)(x-y)^p+y\psi(x/2)\frac{(x-y)^p}{2^p}
  +(x-y)\psi(x/2\vee y)|x/2-y|^p - x \Vlyap(x,y)\\
  & = -\Vlyap(x,y)\left[-\frac{\psi'(x)}{\psi(x)} +
    x - \frac{\psi(x/2)}{\psi(x)}\frac y{2^p}-\frac{\psi(x/2\vee y)}{\psi(x)}
  (x-y)^{1-p}|x/2-y|^p\right]\\
  & \leq -\Vlyap(x,y) \left[-\frac{\psi'(x)}{\psi(x)}
     + x\left(1-\frac{\psi(x/2)}{\psi(x)}\varphi\left(\frac yx\right)\right)  \right] \\
  & \leq - \Vlyap(x,y)\left[-\frac{\psi'(x)}{(1+\alpha)} + 
  x\left(1-(1+\alpha)M\right)  \right] \\
& \qquad =
\begin{cases}
 -\Vlyap(x,y)
 \left[\frac{2\alpha}{x_{0}(1+\alpha)}
       + x\left(1-(1+\alpha)M-\frac{2\alpha}{x_{0}^2(1+\alpha)}\right)
 \right]
 & \text{when $x\leq x_{0}$},\\
 -\Vlyap(x,y)x\left(1-(1+\alpha)M\right)
 & \text{when $x\geq x_{0}$.}
\end{cases}
\end{align*}
Finally, as soon as $(1+\alpha)M<1$, one has
\begin{equation*}
 \gen \Vlyap (x,y) \leq -\lambda_{\alpha,x_{0}} \Vlyap(x,y) \qquad \forall x,y>0
\end{equation*}
with
\begin{equation*}
 \lambda_{\alpha,x_{0}}=
 \min\left(\frac{2\alpha}{x_{0}(1+\alpha)},x_{0}(1-(1+\alpha)M)\right).
\end{equation*}

By direct computations, one gets that the best possible choice of parameters 
$\alpha$ and $x_{0}$ is $\alpha=1/\sqrt M-1\sim0.0934$ and $x_{0}=\sqrt 2$.
We obtain finally, for any $x,y>0$,
\begin{equation}
\label{eq:uped} \gen \Vlyap (x,y) \leq -\lambda \Vlyap(x,y),
\end{equation}
with $\lambda=\lambda_{\alpha,x_{0}}=\sqrt 2(1-\sqrt M)\sim 0.1208$.
Hence, directly from \eqref{eq:uped}, for any $x,y>0$
\begin{equation}
 \label{eq:estW}   \dE_{x,y}[ \Vlyap(X_t,Y_t)] \leq e^{-\lambda t} \Vlyap(x,y).
\end{equation}
Immediate manipulations lead to the following estimate.

\begin{prop}\label{prop:demi}
 Let 
\begin{equation}\label{eq:constant}
M=\frac{\sqrt{2}(3+\sqrt{3})}8 \sim 0.84 
\quad\text{and}\quad
\lambda=\sqrt 2(1-\sqrt M)\sim 0.12.  
\end{equation}
Then, for any $x,y>0$, one has
\begin{equation}\label{eq:contraction-racine}
\dE_{x,y}\PAR{|X_t-Y_t|^{1/2}} \leq 
\frac1{\sqrt M}e^{-\lambda t} |x-y|^{1/2}. 
\end{equation}
Moreover, for any initial conditions $(X_0,Y_0)$ and any $t\geq t_0>0$, 
\begin{equation}\label{eq:borne-racine}
\dE\PAR{|X_t-Y_t|^{1/2}} \leq 
\sqrt{\frac{2\sqrt 2+4 t_0^{-1}}{ M}} e^{-\lambda (t-t_0)}. 
\end{equation}
\end{prop}

\begin{proof}
Equation~\eqref{eq:contraction-racine} is a straightforward consequence of 
Equation~\eqref{eq:uped} since $\psi(\dR^+)$ is the interval $[1,1+\alpha]$. 
As a consequence, for any $t\geq t_0>0$, one has 
\begin{align*}
\dE\PAR{|X_t-Y_t|^{1/2}} &\leq 
\frac1{\sqrt M}e^{-\lambda(t-t_0)} \dE\PAR{|X_{t_0}-Y_{t_0}|^{1/2}}\\
&\leq 
\frac{e^{\lambda t_0}}{\sqrt M}\sqrt{\dE(X_{t_0})+\dE(Y_{t_0})} e^{-\lambda t}.
\end{align*}
Then, Lemma~\ref{le:momest} provides the estimate~\eqref{eq:borne-racine}.
\end{proof}

\begin{figure}[htbp]
  \centering
  \includegraphics[width=0.55\linewidth]{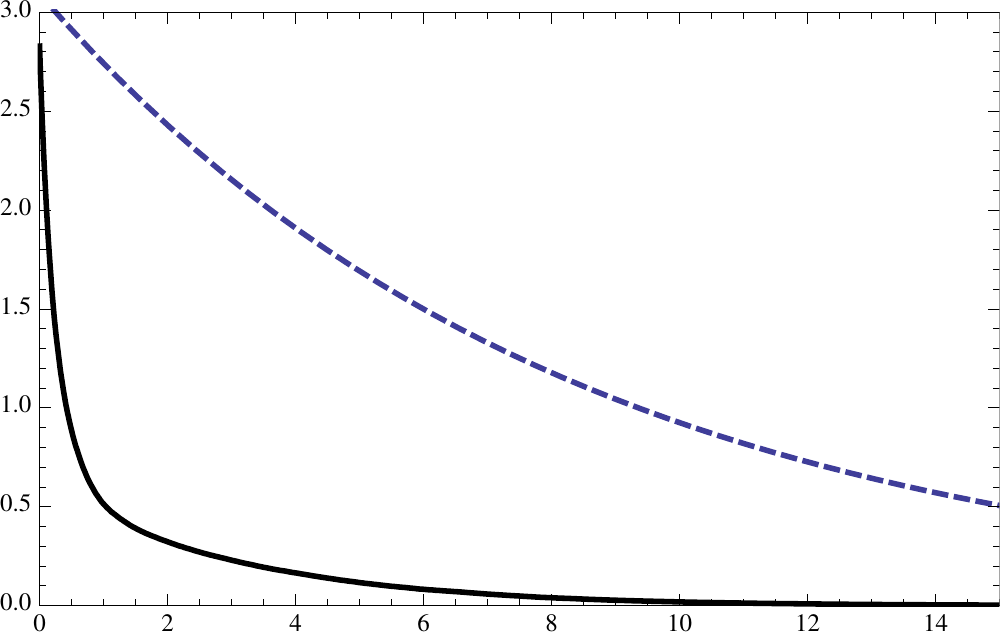}
  \caption{The ``true'' function $\dE_{x,y}\PAR{|X_t-Y_t|^{1/2}}$ (solid, black; 
  by Monte Carlo method with $m=10\ 000$ copies) and the bound given 
  in Equation \eqref{eq:contraction-racine} (dashed, blue), for $x=2$ and $y=10$.}
  \label{fig:conv12}
\end{figure}

\begin{rem} The upper bound obtained in Proposition \ref{prop:demi} is 
compared graphically to the ``true'' function 
$t\mapsto\dE_{x,y}\PAR{|X_t-Y_t|^{1/2}}$ in Figure \ref{fig:conv12}. By 
linear regression on this data, one gets that the exponential speed 
of convergence of this function is on the order of $0.4$.\\ 
Note also that this method can be adapted to any
  $V_p$ with $0<p<1$, giving even better (but less explicit) speed of
  convergence for some $p\neq\frac12$: we estimated numerically that
  the best value for $ \lambda$ would be approximately $0.1326$,
  obtained for $p$ close to $2/3$. 
\end{rem}

We may now deduce from Proposition~\ref{prop:demi} estimates for the 
Wasserstein distance between the laws of~$X_{t}$ and~$Y_{t}$. 

\begin{thm}\label{thm:convW1} Let $p\geq1$. Then, for any $t_0>0$ and 
any $\theta\in(0,1)$, there exists $C(p,t_0,\theta)<+\infty$ such that,  
for any initial conditions $(X_{0}, Y_{0})$ and for all $t\geq t_0$,
\[
\dE\PAR{\ABS{X_t-Y_t}^p}\leq
 C(p,t_0,\theta) \exp\PAR{- \lambda\theta t}
\]
 where $\lambda$ is defined by~\eqref{eq:constant}.
\end{thm}

\begin{proof}[Proof of Theorem~\ref{thm:convW1}] Let $p\geq1$. 
For any $0<\theta<1$, H\"older's inequality gives, for any $t\geq 0$,  
\[
  \dE\PAR{\abs{X_t - Y_t}^{p}} \leq 
  \SBRA{\dE\PAR{\abs{X_t - Y_t}^{\frac{2p-\theta}{2(1-\theta)}}}}^{1-\theta} 
  \SBRA{\dE\PAR{\abs{X_t - Y_t}^{1/2}}}^{\theta}.
\]
Thanks to Lemma~\ref{le:momest} and the inequality 
$(a+b)^q\leq2^{q-1}(a^q+b^q)$, when $q\geq1$, one gets 
\begin{align*}
\SBRA{\dE\PAR{\abs{X_t - Y_t}^{\frac{2p-\theta}{2(1-\theta)}}}}^{1-\theta}
&\leq \PAR{2^{\frac{2p-\theta}{2(1-\theta)}} 
M_{\frac{2p-\theta}{2(1-\theta)},t}}^{1-\theta}\\
&\leq 2^{p-\theta/2} 
\PAR{\sqrt{\frac{2p-\theta}{1-\theta}}+\frac{2p-\theta}{(1-\theta)t}}^{p-\theta/2}.
\end{align*}
Then, it suffices to use Equation~\eqref{eq:borne-racine} to conclude the proof
with
\[
 {C(p,t_0,\theta)}= 
 \frac{2^{p}}{M^{\theta/2}}\PAR{\sqrt{\frac{2p-\theta}{1-\theta}}
+\frac{2p-\theta}{(1-\theta)t_0}}^{p-\theta/2}
 \PAR{\sqrt 2+\frac{2}{t_0}}^{\theta/2} 
 e^{\lambda \theta t_0}\,.
\]
\end{proof}
Since $W_{p}(\law(X_{t}),\law(Y_{t}))^p\leq 
  \dE\PAR{\abs{X_t - Y_t}^{p}}$, Theorem~\ref{th:wasserstein} is a direct 
  consequence of this result. 

\begin{rem}\label{expcont}
Let us remark that we can obtain "contraction's type bounds"  using 
Equation~\eqref{eq:contraction-racine} instead of \eqref{eq:borne-racine} : 
for any $p\geq 1$, any $0<\theta<1$, any $t\geq 0$ and any $x,y>0$,
\[
\dE_{x,y}\PAR{\abs{X_t-Y_t}^p}\leq \frac{2^{p-\theta/2} }{M^{\theta/2}}
\PAR{\sqrt{\frac{2p-\theta}{1-\theta}}+\frac{2p-\theta}{(1-\theta)t}}^{p-\theta/2} \, e^{-\lambda\theta t}\abs{x-y}^{\theta/2}.
\]
We then obtain that if $\nu$ and $\tilde\nu$ have finite $\theta/2$-moments then for $p\geq1$
and $t\geq 0$,
\[ W_p(\nu P_t,\tilde \nu P_t)\leq \PAR{\frac{2^{p-\theta/2} }{M^{\theta/2}}
\PAR{\sqrt{\frac{2p-\theta}{1-\theta}}+\frac{2p-\theta}{(1-\theta)t}}^{p-\theta/2}}^{\frac1p} e^{-\frac{\lambda \theta}{p} t} \,W_{\frac\theta 2}(\nu,\tilde\nu)^{\frac{2}{p\theta}}\,.
\]
which still allows a control by some Wasserstein "distance" (in fact, this is not a distance, 
since $\theta/2<1$) of the initial measures.
\end{rem}

  \begin{rem}
    We estimated numerically the exponential speed of convergence:
\begin{itemize}
\item of the function $t\mapsto \dE_{x,y}\PAR{\abs{X_t - Y_t}}$ for 
the Wasserstein coupling (by Monte Carlo method and linear regression). 
It seems to be on the order of $0.5$ (we obtained $0.48$ for $x=2$, 
$y=10$, $m=10\ 000$ copies, and linear regression between times $2$ 
and $10$);
\item of the Wasserstein distance $t\mapsto W_1(\delta_xP_t,\delta_yP_t)$, 
using the explicit representation of this distance for measures on $\mathbb R$ 
to approximate it by the $L^1$-distance between the two empirical 
quantile functions. It is on the order of $1.6$ (we get $1.67$ for $x=2$, 
$y=0.5$, $m=1\ 000\ 000$ copies, and linear regression on $20$ points 
until time $4$).
 \end{itemize} 
In conclusion, our bound from Theorem \ref{thm:convW1} seems reasonable 
(at least when compared to those given by \cite{RR96}, see 
section \ref{sub:A bound via small sets} below), but is still off by a factor 
of $4$ from the true exponential speed of convergence. 
Since the coupling itself seems to converge approximately $3$ times 
slower than the true Wasserstein distance, one probably needs to find 
another coupling to get better bounds. 
\end{rem}

\begin{rem}
Let us end this section by advertising on a parallel work 
by B. Cloez \cite{cloez} who uses a completely different approach, 
based on a particular Feynman-Kac formulation, to get some related results.
\end{rem}
  
\section{Exponential convergence in total variation distance}\label{se:variation}

In this section, we provide the proof of Theorem~\ref{th:variation} and 
we compare our estimate to the ones that can be deduced from \cite{RR96}. 

\subsection{From Wasserstein to total variation estimate}

The fact that the evolution is partially deterministic makes the study 
of convergence in total variation quite subtle. Indeed,
the law $\delta_xP_t$ can be written as a mixture of a Dirac 
mass at $x+t$ (the first jump time has not occured) and an 
absolutely continuous measure (the process jumped at some 
random time in $(0,t)$): 
\[
\delta_xP_t=p_t(x)\delta_{x+t}+
(1-p_t(x))\cL(X_t\vert X_0=x,T_1\leq t)
\]
where, according to Equation~\eqref{eq:def-fx}, 
\begin{equation}\label{eq:pt}
p_t(x):=\dP_x(T_1>t)=e^{-t^2/2-xt}.
\end{equation}
This implies that the map $y\mapsto \TVNRM{\delta_xP_t-\delta_yP_t}$ 
is not continuous at point $x$ since one has, for any $y\neq x$,
\[
\TVNRM{\delta_xP_t-\delta_yP_t}\geq 
p_t(x)\vee p_t(y)=e^{-t^2/2-(x\wedge y) t}.
\] 
Nevertheless, one can hope that 
\[
\TVNRM{\delta_xP_t-\delta_yP_t}
\underset{y\to x}{\sim}
p_t(x).
\]
The lemma below makes this intuition more precise. 
\begin{lem}\label{lem:TV}
Let $t\geq\varepsilon\geq x-y> 0$. There exists a coupling 
$\PAR{{(X_t)}_{t\geq 0},{(Y_t)}_{t\geq 0}}$ of two TCP processes 
driven by \eqref{eq:G1} and starting at $(x,y)$ such that the probability 
$\dP(X_s=Y_s,\ s\geq t)$ is larger than 
\begin{align} 
\notag
 q_{t}(x,y) & =
 \left(\int_{x-y}^t\! (f_x(s)\wedge f_y(s-x+y))\,ds\right) 
 p_{x-y}\PAR{\frac{x+t}{2}}\\
 \label{eq:bou-q}
& \geq (p_\varepsilon(x)-p_t(x)-2\varepsilon\alpha(x))
p_{\varepsilon}\PAR{\frac{x+t}{2}},
\end{align}
where $f_x$ is defined in Equation~\eqref{eq:def-fx} and 
$\alpha(x):=\int_0^\infty\! e^{-u^2/2-ux}\,du$.

Moreover, for any $x_0>0$ and $\varepsilon>0$, let us define 
\begin{equation}\label{eq:def-A}
A_{x_0,\varepsilon}
=\BRA{(x,y)\,:\, 0\leq x,y\leq x_0,\ \ABS{x-y}\leq \varepsilon}.
\end{equation}
Then, 
\[
\inf_{(x,y)\in A_{x_0,\varepsilon}} q_t(x,y)\geq 
\exp\PAR{-\varepsilon^2-\frac{3x_0+t}{2}\varepsilon}
-e^{-t^2/2}-\sqrt{2\pi}\varepsilon. 
\]
\end{lem}

\begin{proof}
 The idea is to construct an explicit coalescent coupling starting 
 from $x$ and $y$. The main difficulty comes from the fact that 
 the jump are deterministic. Assume for simplicity 
 that $y<x$. Let us denote by ${(T^x_k)}_{k\geq 1}$ and 
 ${(T^y_k)}_{k\geq 1}$ the jump times of the two processes. If 
\begin{equation}\label{eq:temps}
 T^x_1=T^y_1+x-y
 \quad\text{and}\quad 
 T^y_2-T^y_1\geq x-y
 \end{equation}
 then the two paths are at the same place at time $T^x_1$ since in this 
 case
\[
 X_{T^x_1}=\frac{x+T^x_1}{2}=
 \frac{y+T^y_1}{2}+T^x_1-T^y_1=Y_{T^x_1}.
\] 
The law of $T^x_1$ has a density $f_x$ given by~\eqref{eq:def-fx}.
 As a consequence, the density of $T^1_y+x-y$ is given by 
 $s\mapsto f_y(s-x+y)$. The best way of obtaining the first equality in 
 \eqref{eq:temps} before a time $t\geq x-y$ is to realize an optimal coupling
 of the two continuous laws of $T_1^x$ and $T_1^y+x-y$. This makes these
 random variables equal with probability  
\[
I(x,y,t)=\int_{x-y}^t\! (f_x(s)\wedge f_y(s-x+y))\,ds. 
\]
Assume now that $0\leq x-y\leq \varepsilon\leq t$. 
For any $s\geq x-y$,  one has
\begin{align*}
 f_y(s-x+y)&=(2y-x+s) \exp\PAR{-\frac{1}{2}(s-x+y)^2-y(s-x+y)}\\
 &=(s+x-2(x-y))\exp\PAR{-\frac{s^2}{2}-xs+(2s+x-3(x-y)/2)(x-y)}\\
 &\geq (s+x-2\varepsilon)\exp\PAR{-\frac{s^2}{2}-xs}.
 \end{align*}
As a consequence, if $0\leq x-y\leq \varepsilon\leq t$,
\begin{align*}
I(x,y,t)&\geq 
\int_{x-y}^t\!(s+x-2\varepsilon)\exp\PAR{-\frac{s^2}{2}-xs} \,ds \\
&\geq p_{\varepsilon}(x)-p_t(x)-2\varepsilon \alpha(x)
\end{align*}
where $p_t(x)$ is defined in \eqref{eq:pt} and 
$\alpha(x)=\int_0^\infty\! e^{-u^2/2-ux}\,du$.

Finally, one has to get a lower bound for the probability of the set 
$\BRA{T^y_2-T^y_1\geq x-y}$. For this, we notice that 
\[
z\in[0,+\infty) \mapsto \dP(T^z_1\geq s)=p_s(z)=e^{-s^2/2-s z}
\]
is decreasing and that $Y_{T^y_1}\leq (y+t)/2$ as soon as $T^y_1\leq t$. 
As a consequence, 
\[
\inf_{0\leq z\leq (x+t)/2} \dP(T^z_1\geq s)
\geq p_{x-y}\PAR{\frac{x+t}{2}}
\geq p_{\varepsilon}\PAR{\frac{x+t}{2}}.
\]
This provides the bound \eqref{eq:bou-q}. The uniform lower bound on the set 
$A_{x_0,\varepsilon}$ is a direct consequence of the previous one.
\end{proof}

Let us now turn to the proof of Theorem~\ref{th:variation}.
\begin{proof}[Proof of Theorem~\ref{th:variation}]
We are looking for an upper bound for the total variation distance between 
$\delta_xP_t$ and $\delta_yP_t$ for two arbitrary initial conditions $x$ and $y$. 
To this end, let us consider the following coupling: during a time $t_1<t$ we 
use the Wasserstein coupling. Then during a time $t_2=t-t_1$ we try to stick the 
two paths using Lemma~\ref{lem:TV}. Let $\varepsilon>0$ and $x_0>0$ be 
as in Lemma~\ref{lem:TV}. If, after the time $t_1$, one has 
\[
(X_{t_1},Y_{t_1})\in A_{\varepsilon,x_0}
\]
where $A_{\varepsilon,x_0}$ is defined by \eqref{eq:def-A}, then 
the coalescent coupling will work with a probability greater than  
the one provided by Lemma~\ref{lem:TV}. As a consequence, the 
coalescent coupling occurs before time $t_1+t_2$ with a probability 
greater than 
\[
\dP((X_{t_1},Y_{t_1})\in A_{\varepsilon, x_0})
\inf_{(x,y)\in A_{\varepsilon,x_0}}q_{t_2}(x,y). 
\]
Moreover, 
\[
\dP\PAR{(X_{t_1},Y_{t_1})\notin A_{\varepsilon,x_0}}
\leq \dP(X_{t_1}\geq x_0)+\dP(Y_{t_1}\geq x_0)
+\dP\PAR{\ABS{X_{t_1}-Y_{t_1}}\geq \varepsilon}.
\]
>From the deviation bound~\eqref{eq:deviation-t}, 
we get that, for any $x_0\geq 2e(1+1/t_1)$,  
\[
\dP(X_{t_1}\geq x_0) \leq e^{-a(t_1)x_0}
\quad\text{with}\quad
a(t_1)=\frac{t_1}{2e(t_1+1)}.
\]
The estimate of Proposition~\ref{prop:demi} concerning the Wasserstein 
coupling ensures that 
\[
 \dP(\ABS{X_{t_1}-Y_{t_1}}\geq \varepsilon)
   \leq \frac{C}{\sqrt{\varepsilon}} e^{-\lambda t_1}
 \quad\text{with} \quad
 C= \sqrt{\frac{2\sqrt 2+4t_0^{-1}}{M}} e^{\lambda t_0},
\]
for any $0<t_0\leq t_1$.
As a consequence, the total variation distance between $\delta_{x} P_{t}$ and 
$\delta_{y}P_{t}$ is smaller than 
\[
D:=1-e^{-\varepsilon^2-\frac{3x_0+t_2}{2}\varepsilon}
+e^{-t_2^2/2}+\sqrt{2\pi}\varepsilon+2e^{-a(t_1)x_0}+\frac{C}{\sqrt{\varepsilon}}
e^{-\lambda t_1}.
\]
In order to get a more tractable estimate, let us assume that $t_2\leq x_0$ and use that $1-e^{-u}\leq u$ to get 
\[
D\leq \varepsilon^2+2\varepsilon x_0
+e^{-t_2^2/2}
+\sqrt{2\pi}\varepsilon
+2e^{- a(t_{0})x_0}
+\frac{C}{\sqrt{\varepsilon}}e^{-\lambda t_1}.
\]
Finally let us set 
\[
t_2=\sqrt{2\log(1/\varepsilon)},\quad 
t_1=\frac{3}{2\lambda}\log(1/\varepsilon),\quad
x_0=\frac1{a(t_{0})}\log (1/\varepsilon).
\]
Obviously, for $\varepsilon$ small enough, $x_0\geq \max(t_2,2e(1+1/t_1))$ 
and $t_1 \geq t_{0}$. Then, one gets that 
\[
\NRM{\delta_xP_{t_1+t_2}-\delta_yP_{t_1+t_2}}_\mathrm{TV}\leq  
\frac2{a(t_{0})} \varepsilon\log (1/\varepsilon) 
+(3+C+\sqrt{2\pi}+\varepsilon)\varepsilon. 
\]
One can now express $\varepsilon$ as a function of $t_1$ to get that
there exists $K=K(t_{0})>0$ such that  
\[
\NRM{\delta_xP_{t_1+t_2}-\delta_yP_{t_1+t_2}}_\mathrm{TV}
\leq K(1+t_1) e^{-\frac{2\lambda}{3}t_1}. 
\] 
Since $t_2= \sqrt{(4\lambda/3)t_1}$, one gets that 
\[
\NRM{\delta_xP_t-\delta_yP_t}_\mathrm{TV}
\leq K(1+t)e^{K\sqrt t} e^{-\frac{2\lambda}{3}t}. 
\] 
This provides the conclusion of Theorem~\ref{th:variation} when 
the initial measures are Dirac masses. The generalisation is straightforward. 
\end{proof}

\subsection{A bound via small sets} 
\label{sub:A bound via small sets}
We describe here briefly the approach of \cite{RR96} and compare it
with the hybrid Wasserstein/total variation coupling described above.
The idea is once more to build a successful coupling between two
copies $X$ and $Y$ of the process. In algorithmic terms, the approach
is the following:
\begin{itemize}
    \item let  $X$ and $Y$ evolve independently until they both reach a given set $C$,
    \item once they are in $C$, try to stick them together,
    \item repeat until the previous step is successful.
\end{itemize}
To control the  time to come back to $C\times C$, \cite{RR96} advocates
an approach via a Lyapunov function.
The second step works with positive probability
if the set is ``pseudo-small'', i.e.\ if one can find a time $t^\star$, an $\alpha>0$ and probability measures $\nu_{xy}$
\begin{equation}
  \label{eq=coupled_minorization}
  \forall x, y \in C^2, \quad
  \law(X_{t^\star} | X_0 = x) \geq \alpha \nu_{xy} \text{ and }
  \law(X_{t^\star} | X_0 = y) \geq \alpha \nu_{xy}.
\end{equation}

 The convergence result can be stated as follows.
\begin{thm}[\cite{RR96}, Theorem 3, Corollary 4 and Theorem 8]\footnote{In fact, in \cite{RR96}, the result is given with $A$ instead of $A'$ in the upper bound. Joaquin Fontbona pointed out to us that Lemma 6 from \cite{RR96} has to be corrected, adding the exponential term $e^{\delta t^\star}$ to the estimate. We thank him for this remark.}
    Suppose that there exists a set $C$, a function $V\geq 1$, and
    positive constants $\lambda, \Lambda$ such that
    \begin{equation}
      \label{eq=lyapSimpl}
      L V \leq - \lambda V + \Lambda \ind_{C}.
    \end{equation}
    Suppose that $\delta = \lambda - \Lambda/\inf_{x\notin C} V >0$.
         Suppose additionally that $C$ is pseudo-small, with constant $\alpha$.

     Then for $A = \frac{\Lambda}{\delta} + e^{- \delta t^\star} \sup_{x\in C} V$, $A' = Ae^{\delta t^\star}$,
     and for any $r < 1/t^\star$,
    \[
    \TVNRM{\mathcal{L}(X_t) - \mu} \leq
    (1 - \alpha)^{\PENT{rt}} + e^{-\delta(t - t^\star)} (A')^{\PENT{rt} - 1}\dE V(X_0).
    \]
\end{thm}
If $A'$ is finite, this gives exponential convergence: just choose $r$ small
enough so that $(A')^{rt} e^{-\delta t}$ decreases exponentially fast.

To compute explicit bounds we have to make choices for $C$ and $V$ and
estimate the corresponding value of $\alpha$. Our best efforts 
for the case of the TCP process only give
decay rates of the order  $10^{-14}$. We believe this order cannot be substantially
improved even by fine-tuning $C$ and $V$.

\section{Two other models}\label{sec:examples}

This section is devoted to the study of two simple PDMPs. The first one is 
a simplified version of the TCP process where the jump rate is assumed 
to be constant and equal to~$\lambda$. It has been studied with different 
approaches: PDE techniques (see \cite{PR,Lau-Per}) or probabilistic tools
(see \cite{LL,MR2426601,CMP10}). The second one is studied in \cite{RT}. 
It can also be checked that our method gives sharp bounds for
the speed of convergence to equilibrium of the PDMP which appears in
the study of a penalized bandit algorithm (see Lemma 7 in \cite{MR2386736}).

\subsection{The TCP model with constant jump rate}

In this section we investigate the long time behavior of the TCP process 
with constant jump rate given by its infinitesimal generator: 
\[
Lf(x)=f'(x)+\lambda (f(x/2)-f(x))
\quad(x\geq 0). 
\]
The jump times of this process are the ones of a homogeneous Poisson 
process with intensity $\lambda$. The convergence in Wasserstein 
distance is obvious. 
\begin{lem}[\cite{PR,CMP10}]\label{wascoup}
For any $p\geq 1$, 
\begin{equation}\label{eq:Wasserstein-coupling}
W_p(\delta_xP_t,\delta_yP_t)\leq \ABS{x-y} e^{-\lambda_pt}
\quad\text{with}\quad
\lambda_p=\frac{\lambda(1-2^{-p})}{p}. 
 \end{equation}
\end{lem}
\begin{rem}
 The case $p=1$ is obtained in \cite{PR} by PDEs estimates using the 
 following alternative formulation of the Wasserstein distance on $\dR$. 
 If the cumulative distribution functions of the two probability 
 measures $\nu$ and $\tilde \nu$ are $F$ and $\tilde F$ then 
 \[
W_1(\nu,\tilde \nu)=\int_{\dR}\! \vert F(x)-\tilde F(x)\vert \,dx. 
\] 
\end{rem}
The general case $p\geq 1$ is obvious from the probabilistic point of view: 
choosing the same Poisson process ${(N_t)}_{t\geq 0}$ to drive 
the two processes provides that the two coordinates jump simultaneously 
and 
\[
\ABS{X_t-Y_t}=\ABS{x-y}{2}^{-N_t}. 
\]
As a consequence, since the law of $N_t$ is the Poisson distribution with 
parameter $\lambda t$, one has 
\[
\dE_{x,y}\PAR{\ABS{X_t-Y_t}^p}
=\ABS{x-y}^p \dE\PAR{2^{-pN_t}}
=\ABS{x-y}^p e^{-p\lambda_pt}.
\]
This coupling turns out to be sharp. Indeed, one can compute explicitly 
the moments of~$X_t$ (see \cite{LL,MR2426601}): for every 
$n\geq0$, every $x\geq0$, and every $t\geq0$,
\begin{equation}\label{eq:momlop}
  \dE_x((X_t)^n)=\frac{n!}{\prod_{k=1}^n\theta_k}+%
  n!\sum_{m=1}^n\biggr(\sum_{k=0}^m\frac{x^k}{k!}%
  \prod_{\substack{j=k\\j\neq m}}^n\frac{1}{\theta_j-\theta_m}\biggr)e^{-\theta_mt},
\end{equation}
where $\theta_n=\lambda(1-2^{-n})=n\lambda_n$ for any 
$n\geq1$. Obviously, assuming for example that $x>y$,
\begin{align*}
W_n(\delta_xP_t,\delta_yP_t)^n
&\geq \dE_x((X_t)^n)-\dE_y((Y_t)^n)\\
&\underset{t\to \infty}{\sim}
 n!\biggr(\sum_{k=0}^n\frac{x^k-y^k}{k!}%
  \prod_{\substack{j=k}}^{n-1}\frac{1}{\theta_j-\theta_n}\biggr)e^{-\theta_nt}.
\end{align*}
As a consequence, the rate of convergence in 
Equation~\eqref{eq:Wasserstein-coupling} is optimal for any $n\geq 1$.

Nevertheless this estimate for the Wasserstein rate of convergence 
does not provide on its own any information about the total variation 
distance between $\delta_xP_t$ and $\delta_yP_t$. It turns out that 
this rate of convergence is the one of the $W_1$ distance. This is 
established by Theorem~1.1 in \cite{PR}. It can be reformulated in our
setting as follows.  
\begin{thm}[\cite{PR}]
Let $\mu$ be the invariant measure of $X$. For any measure $\nu$ 
with a finite first moment and $t\geq 0$, 
\[
\TVNRM{\nu P_t -\mu} \leq 
e^{-\lambda t/2}\PAR{3\lambda W_1(\nu,\mu)+\TVNRM{\nu-\mu}}.
\]
\end{thm}

Let us provide here an improvement of this result by a probabilistic argument.

\begin{prop}
\label{pr:constantTCP}
For any $x,y\geq 0$ and $t\geq 0$, 
\begin{equation}\label{eq:constant-TV-xy}
\TVNRM{\delta_xP_t-\delta_yP_t}\leq 
\lambda e^{-\lambda t/2}\ABS{x-y}+e^{-\lambda t}. 
\end{equation}
As a consequence, for any measure $\nu$ with a finite first moment 
and $t\geq 0$, 
\begin{equation}\label{eq:constant-TV-ergo}
 \TVNRM{\nu P_t-\mu}\leq 
\lambda e^{-\lambda t/2}W_1(\nu,\mu)+
 e^{-\lambda t} \TVNRM{\nu-\mu} .
\end{equation}
\end{prop}

\begin{rem}
  Note that the upper bound obtained in Equation \eqref{eq:constant-TV-xy} is non-null even for $x=y$. This is due to the persistence of a Dirac mass at any time, which implies that taking $y$ arbitrarily close to $x$ for initial conditions does not make the total variation distance arbitrarily small, even for large times.
\end{rem}

\begin{proof}[Proof of Proposition \ref{pr:constantTCP}]
The coupling is a slight modification of the one used to control Wasserstein distance.  
The paths of ${(X_s)}_{0\leq s\leq t}$ and ${(Y_s)}_{0\leq s\leq t}$ 
starting respectively from $x$ and $y$ are determined by their jump 
times ${(T^X_n)}_{n\geq 0}$ and ${(T^Y_n)}_{n\geq 0}$ up to time $t$. These
sequences have the same distribution than the jump times of a Poisson 
process with intensity $\lambda$. 

Let ${(N_t)}_{t\geq 0}$ be a Poisson process with intensity $\lambda$ 
and ${(T_n)}_{n\geq 0}$ its jump times with the convention $T_0=0$.
Let us now construct the jump times of $X$ and $Y$. Both  
processes make exactly $N_t$ jumps before time $t$. If $N_t=0$, then 
\[
X_s=x+s\quad \text{and}\quad Y_s=y+s\quad \text{for }0\leq s\leq t.
\]
Assume now that $N_t\geq 1$. The $N_t-1$ first jump times of $X$ and $Y$ 
are the ones of ${(N_t)}_{t\geq 0}$: 
\[
T^X_k=T^Y_k=T_k \quad 0\leq k\leq N_t-1.  
\] 
In other words, the coupling used to control Wassertein distance (see Lemma \ref{wascoup})
acts until the penultimate jump time $T_{N_t-1}$. At that time, we have 
\[
X_{T_{N_t-1}}-Y_{T_{N_t-1}}=\frac{x-y}{2^{N_t-1}}.
\]
Then we have to define the last jump time for each process. If they are 
such that 
\[
T^X_{N_t}=T^Y_{N_t}+X_{T_{N_t-1}}-Y_{T_{N_t-1}}
\]
then the paths of $X$ and $Y$ are equal on the 
interval $(T^X_{N_t},t)$ and can be chosen to be equal for any time 
larger than $t$.  

Recall that conditionally on the event $\BRA{N_t=1}$, the law of 
$T_1$ is the uniform distribution on $(0,t)$. More generally, if  $n\geq 2$, 
conditionally on the set $\BRA{N_t=n}$, the law of the penultimate jump 
time $T_{n-1}$ has a density $s\mapsto n(n-1)t^{-n}(t-s)s^{n-2}\ind_{(0,t)}(s)$ 
and conditionally on the event $\BRA{N_t=n, T_{n-1}=s}$, the law of $T_n$ 
is uniform on the interval~$(s,t)$. 

Conditionally on $N_t=n\geq 1$ and $T_{n-1}$, $T^X_n$ and $T^Y_n$ 
are uniformly distributed on $(T_{n-1},t)$ and can be chosen such that  
\begin{align*}
&\dP\PAR{T^X_n=T^Y_n+\frac{x-y}{2^{n-1}}
\,\Big\vert\, N^X_t=N^Y_t=n,\,T^X_{n-1}=T^Y_{n-1}=T_{n-1}}\\
&\quad\quad=\PAR{1-\frac{|x-y|}{2^{n-1}(t-T_{n-1})}}\vee 0
\geq 1-\frac{|x-y|}{2^{n-1}(t-T_{n-1})}.
\end{align*}
This coupling provides that 
\begin{align*}
\TVNRM{\delta_xP_t-\delta_yP_t}&\leq 
1-\dE\SBRA{\PAR{1-\frac{|x-y|}{2^{N_t-1}(t-T_{N_t-1})}}\ind_\BRA{N_t\geq1}}\\
&\leq e^{-\lambda t}+|x-y|
\dE\PAR{\frac{2^{-N_t+1}}{(t-T_{N_t-1})}\ind_\BRA{N_t\geq 1}}.
\end{align*}
For any $n\geq 2$, 
\[
\dE\PAR{\frac{1}{t-T_{N_t-1}}\Big\vert N_t=n}
=\frac{n(n-1)}{t^n}\int_0^t\! u^{n-2}\,du=\frac{n}{t}.
\]
This equality also holds for $n=1$. Thus we get that 
\[
\dE\PAR{\frac{2^{-N_t+1}}{(t-T_{N_t-1})}\ind_\BRA{N_t\geq 1}}
=\frac{1}{t}\dE\PAR{N_t 2^{-N_t+1}}=\lambda e^{-\lambda t/2},
\]
since $N_t$ is distributed according to the Poisson law with 
parameter $\lambda t$. This provides the 
estimate \eqref{eq:constant-TV-xy}.

To treat the case of general initial conditions and to get \eqref{eq:constant-TV-ergo}, we combine the coupling between the dynamics constructed above with the choice of the coupling of the initial measures $\mu$ and $\nu$ as a function of the underlying Poisson process ${(N_t)}_{t\geq0}$: the time horizon $t>0$ being fixed, if $N_t=0$, one chooses for $\law(X_0,Y_0)$ the optimal total variation coupling of $\nu$ and $\mu$; if $N_t\geq1$, one chooses their optimal Wasserstein coupling. One checks easily that this gives an admissible coupling, in the sense that its first (resp. second) marginal is a constant rate TCP process with initial distribution $\nu$ (resp. $\mu$). And one gets with this construction, using the same estimates as above in the case where $N_t\geq1$:
\begin{multline*}
 \dP\PAR{X_t\neq Y_t}= \dP\PAR{X_t\neq Y_t,N_t\geq 1}
 +\dP\PAR{X_t\neq Y_t,N_t=0}\\
 \leq \dE\PAR{\ABS{X_0-Y_0}\frac{2^{-N_t+1}}{(t-T_{N_t-1})}\ind_\BRA{N_t\geq 1}}+\dP\PAR{X_0\neq Y_0,N_t=0}\\
  = \lambda e^{-\lambda t/2}W_1(\nu,\mu)+
 e^{-\lambda t} \TVNRM{\nu-\mu}\,,
\end{multline*}
which clearly implies \eqref{eq:constant-TV-ergo}. 
\end{proof}

\subsection{A storage model example}\label{sec:storage}

In \cite{RT}, Roberts and Tweedie improve the approach from \cite{RR96} 
via Lyapunov functions and minorization conditions in the specific case 
of stochastically monotonous processes. They get 
better results on the speed of convergence to equilibrium in this
case. 
They give the following example of a storage model as a good illustration 
of the efficiency of their method.
The process ${(X_t)}_{t\geq 0}$ 
on $\dR^+$ is driven by the generator
\[
L f(x)=-\beta x f'(x)+\alpha\int_0^\infty\! (f(x+y)-f(x)) e^{-y}\,dy.
\]
In words, the current stock $X_t$ decreases exponentially at rate $\beta$, and increases at random exponential times by a random (exponential) amount.
Let us introduce a Poisson process ${(N_t)}_{t\geq 0}$ with 
intensity~$\alpha$ and jump times ${(T_i)}_{i\geq 0}$ (with $T_0=0$) and a 
sequence~${(E_i)}_{i\geq 1}$ of independent random variables 
with law~$\cE(1)$ independent of ${(N_t)}_{t\geq 0}$. 
The process ${(X_t)}_{t\geq 0}$ starting from $x\geq 0$ can be 
constructed as follows: for any $i\geq 0$, 
\[
X_t=
\begin{cases}
e^{-\beta (t-T_i)}X_{T_i}&\text{if }T_i\leq t<T_{i+1},\\
e^{-\beta (T_{i+1}-T_i)}X_{T_i}+E_{i+1} &\text{if }t=T_{i+1}. 
\end{cases}
\]

\begin{prop}\label{prop:storage}
For any $x,y\geq 0$ and $t\geq 0$, 
\[
W_p(\delta_xP_t,\delta_yP_t)\leq \ABS{x-y}e^{-\beta t},
\]
and (when $\alpha\neq\beta$)
\begin{equation}\label{eq:storage-TV}
\TVNRM{\delta_xP_t-\delta_yP_t}\leq 
e^{-\alpha t}+\ABS{x-y} \alpha\frac{e^{-\beta t}-e^{-\alpha t}}{\alpha-\beta}.
\end{equation}
Moreover, if $\mu$ is the invariant measure of the process $X$, 
we have for any probability measure $\nu$ with a finite first moment and 
$t\geq 0$,  
\[
\TVNRM{\nu P_t-\mu }\leq 
\TVNRM{\nu -\mu } e^{-\alpha t}+
W_1(\nu,\mu)\alpha\frac{e^{-\beta t}-e^{-\alpha t}}{\alpha-\beta}. 
\]
\end{prop}

\begin{rem}
  In the case $\alpha = \beta$, the upper bound \eqref{eq:storage-TV} 
  becomes
\begin{equation*}
\TVNRM{\delta_xP_t-\delta_yP_t}\leq 
(1 + \ABS{x-y} \alpha t )e^{-\alpha t}.
\end{equation*}
\end{rem}

\begin{rem}[Optimality]
 Applying $L$ to the test function $f(x)=x^n$ allows us to compute recursively 
 the moments of $X_t$. In particular, 
 \[
\dE_x(X_t)=\frac{\alpha}{\beta}+\PAR{x-\frac{\alpha}{\beta}}e^{-\beta t}.
\]
This relation ensures that the rate of convergence for the Wasserstein distance 
is sharp. Moreover, the coupling of total variation distance requires at least one 
jump. As a consequence, the exponential rate of convergence is greater than 
$\alpha$. Thus, Equation~\eqref{eq:storage-TV} provides the optimal rate 
of convergence $\alpha\wedge\beta$. 
\end{rem}

\begin{rem}[Comparison with previous work]
By way of comparison, the original 
method of \cite{RT} does not seem to give these optimal rates. The case
 $\alpha=1$ and $\beta=2$ is treated in this paper (as an illustration of 
 Theorem 5.1), with explicit choices for the various parameters needed 
 in this method. With these choices, in order to get the convergence rate, 
 one first needs to compute the quantity $\theta$ (defined in Theorem 3.1), 
 which turns out to be approximately $5.92$. The result that applies is 
 therefore the first part of Theorem 4.1 (Equation (27)), and the convergence 
 rate is given by $\tilde{\beta}$ defined by Equation (22). The computation 
 gives the approximate value $0.05$, which is off by a factor $20$ 
 from the optimal value $\alpha\wedge \beta = 1$.  
\end{rem}

\begin{proof}[Proof of Proposition~\ref{prop:storage}]
Firstly, consider two processes $X$ and $Y$
starting respectively at $x$ and $y$ and driven by the same 
randomness (\emph{i.e.} Poisson process and jumps). Then the 
distance between $X_t$ and $Y_t$ is deterministic: 
\[
X_t-Y_t=(x-y) e^{-\beta t}. 
\] 
Obviously, for any $p\geq 1$ and $t\geq 0$, 
\[
W_p(\delta_xP_t,\delta_yP_t)\leq |x-y| e^{-\beta t}.
\]
Let us now construct explicitly a coupling at time $t$ to get the upper 
bound \eqref{eq:storage-TV} for the total variation distance. The jump times 
of ${(X_t)}_{t\geq 0}$ and ${(Y_t)}_{t\geq 0}$ are the ones of a 
Poisson process ${(N_t)}_{t\geq 0}$ with intensity $\alpha$ and 
jump times ${(T_i)}_{i\geq 0}$. Let us now construct the jump heights 
${(E^X_i)}_{1\leq i\leq N_t}$ and ${(E^Y_i)}_{1\leq i\leq N_t}$ of 
$X$ and $Y$ until time $t$. If $N_t=0$, no jump occurs. If $N_t\geq 1$, 
we choose $E^X_i=E^Y_i$ for $1\leq i\leq N_t-1$ and
$E^X_{N_t}$ and $E^Y_{N_t}$ in order to maximise the probability 
\[
\dP\PAR{X_{T_{N_t}}+E^X_{N_t}=Y_{T_{N_t}}+E^Y_{N_t}
\big\vert X_{T_{N_t}},Y_{T_{N_t}}}.
\]
This maximal probability of coupling is equal to 
\[
\exp\PAR{-|X_{T_{N_t}}-Y_{T_{N_t}}|}
=\exp\PAR{-|x-y|e^{-\beta T_{N_t}}}
\geq 1-|x-y|e^{-\beta T_{N_t}}.
\]
As a consequence, we get that 
\begin{align*}
\TVNRM{\delta_xP_t-\delta_yP_t}&\leq 
1-\dE\SBRA{\PAR{1-|x-y|e^{-\beta T_{N_t}}}\ind_\BRA{N_t\geq 1}}\\
&\leq e^{-\alpha t}+|x-y|\dE\PAR{e^{-\beta T_{N_t}}\ind_\BRA{N_t\geq 1}}. 
\end{align*}
The law of $T_n$ conditionally on the event $\BRA{N_t=n}$ has the density 
\[
u\mapsto n \frac{u^{n-1}}{t^n}\ind_{[0,t]}(u).
\] 
This ensures that
\[
\dE\PAR{e^{-\beta T_{N_t}}\ind_\BRA{N_t\geq 1}}
=\int_0^1\! e^{-\beta t v}\dE\PAR{N_tv^{N_t-1}}\,dv.
\]
Since the law of $N_t$ is the Poisson distribution with parameter $\lambda t$, one 
has 
\[
\dE\PAR{N_tv^{N_t-1}}=\alpha t e^{\alpha t(v-1)}. 
\]
This ensures that 
\[
\dE\PAR{e^{-\beta N_t}\ind_\BRA{N_t\geq 1}}
=\alpha \frac{e^{-\beta t}-e^{-\alpha t}}{\alpha-\beta}
\] 
which completes the proof. Finally, to get the last estimate, we proceed 
as follows: if $N_t$ is equal to 0, a coupling in total variation of the initial 
measures is done, otherwise, we use the coupling above (the method is exactly the same as 
for the equivalent result in Proposition \ref{pr:constantTCP}, see its proof
for details).  
\end{proof}

\section{The case of diffusion processes}\label{sec:diffusion}

Let us consider the process ${(X_t)}_{t\geq 0}$ on $\dR^d$ solution of 
\begin{equation}\label{eq:SDE}
dX_t=A(X_t)\,dt+\sigma(X_t)\,dB_t, 
\end{equation}
where ${(B_t)}_{t\geq 0}$ is a standard Brownian motion on $\dR^n$, 
$\sigma$ is a smooth function from $\dR^d$ to ${\mathcal M}_{d,n}(\dR)$ 
and $A$ is a smooth function from $\dR^d$
to $\dR^d$. Let us denote by ${(P_t)}_{t\geq 0}$ the semi-group associated 
to ${(X_t)}_{t\geq 0}$. If $\nu$ is a probability measure on $\dR^d$, 
$\nu P_t$ stands for the law of $X_t$ when the law of $X_0$ is $\nu$.  

Under ergodicity assumptions, we are interested in 
getting quantitative rates of convergence of $\cL(X_t)$ to its invariant 
measure in terms of classical distances (Wasserstein distances, total 
variation distance, relative entropy,\ldots).  Remark that if $A$ is not in gradient form (even if $\sigma$ is constant), $X_t$ is not reversible and the invariant measure is usually unknown, so that it is quite difficult to use functional inequalities such as Poincar\'e or logarithmic Sobolev to get a quantitative rate of convergence in total variation or Wasserstein distance (using for example Pinsker's inequality or more generally transportation-information inequality). 
Therefore the only general tool seems to be  Meyn-Tweedie's approach, via small sets and Lyapunov functions, as explained in Section 4.2. 
However, we have seen that in practical examples the resulting estimate can be quite poor. 

The main goal of this short section is to recall the known results establishing the decay in Wasserstein distance and then to propose a strategy to derive control in total variation distance.

\subsection{Decay in Wasserstein distance}

The coupling approach to estimate the decay in Wasserstein distance 
was recently put forward, see \cite{CGM08} and \cite{BGM11} or 
\cite{eberle}. It is robust enough to deal with nonlinear diffusions 
or hypoelliptic ones. In \cite{BGG}, the authors approach the problem 
directly, by differentiating the Wasserstein distance along 
the flow of the SDE. 

Let us gather some of the results in these papers in the following statement.  

\begin{prop}\label{prop:Wasserstein-diff}
\begin{enumerate}
\item Assume that there exists $\lambda>0$ such that for $p>1$
\begin{equation}\label{eq:coercivity}
 \frac{p-1}2(\sigma(x)-\sigma(y))(\sigma(x)-\sigma(y))^t
 +(A(x)-A(y))\cdot(x-y)\leq - \lambda \ABS{x-y}^2,
\quad x,y\in\dR^d. 
\end{equation} 
 Then, for any $\nu,\,\nu'\in\cP_p(\dR^d)$, one has
\begin{equation}\label{eq:conv-delta-delta}
W_p(\nu P_t,\nu' P_t)\leq e^{-\lambda t}W_p(\nu,\nu').  
\end{equation}
\item Assume that for all $x$, $\sigma(x)=\alpha \operatorname{I}_d$ for 
some constant $\alpha$ and that $A$ is equal to $-\nabla U$ where 
$U$ is a $\cC^2$ function such that 
$\operatorname{Hess}(U)\leq -K \operatorname{I}_d$ outside a ball $B(0,r_0)$ 
and $\operatorname{Hess}(U)\leq \rho \operatorname{I}_d$ inside this ball for 
some positive $\rho$. Then there exists $c>1,\alpha>0$ such that
\begin{equation}\label{eq:eberle}
W_1(\nu P_t,\nu' P_t)\leq c\,e^{-\alpha t}W_1(\nu,\nu').  
\end{equation}
\item Suppose that the diffusion coefficient $\sigma$ is constant. Assume that 
$A$ is equal to $-\nabla U$ where $U$ a $\cC^2$ convex function such that 
$\operatorname{Hess}(U)\leq \rho\operatorname{I}_d$ outside a ball $B(0,r_0)$, 
with $\rho>0$. Then there exists an invariant probability measure $\nu_\infty$
and $\alpha>0$ such that
\begin{equation}
\label{eq:bgg}
W_2(\nu P_t,\nu_\infty)\leq e^{-\alpha t}W_2(\nu,\nu_\infty).  
\end{equation}
\end{enumerate}
\end{prop}
\begin{proof}
 The first point is usually proved using a trivial coupling (and it readily extends to 
 $p=1$ in the case of constant diffusion coefficient), namely considering the 
 same Brownian motion for two different solutions of the SDE starting 
 with different initial measures. Note also that, in the case $p=2$, the 
 coercivity condition~\eqref{eq:coercivity} is equivalent to the 
 uniform contraction property~\eqref{eq:conv-delta-delta} for the $W_2$
 metric (see \cite{vRS}). 

The second point  is due in this form to Eberle \cite{eberle}, using reflection 
coupling as presented by \cite{LR}, used originally in this form by Chen and Wang to prove spectral gap estimates (see a nice short proof in \cite[Prop. 2.8]{HSV}) in the reversible case.

Finally, the third part was proved by \cite{BGG} establishing the 
dissipation of the Wasserstein distance by an adequate functional inequality.
\end{proof}

\begin{rem}
\begin{itemize}
\item Let us remark that the contraction property of the two first points 
ensures the existence of an invariant probability measure $\nu_\infty$ 
and an exponential rate of convergence towards this invariant measure 
in Wasserstein distance.
\item Note also that in the hypoelliptic case of a kinetic Fokker-Planck equation:
\[
\begin{cases}
dx_t=v_t dt\\
dv_t=dB_t-F(x_t)dt-V(v_t)dt 
\end{cases}
\]
with $F(x)\sim ax$ and $V(v)\sim bv$ for positive $a$ and $b$, one has 
\[
W_2(\nu P_t,\nu'P_t)\le c\, e^{-\alpha t}W_2(\nu,\nu')
\]
for $c>1$ and a positive $\alpha$ (see \cite{CGM08}).
\end{itemize}
\end{rem}

\subsection{Total variation estimate}

If $\Sigma=0$ in Equation~\eqref{eq:SDE}, the process ${(X_t)}_{t\geq 0}$ 
is deterministic and its invariant measure is a Dirac mass at the unique point 
$\bar x\in\dR^d$ such that $A(\bar x)=0$. As a consequence, for any 
$x\neq\bar x$, 
\[
\TVNRM{\delta_xP_t-\delta_{\bar x}}=1
\quad\text{and}\quad
H(\delta_xP_t\vert\delta_{\bar x})=+\infty.
\]
A non-zero variance is needed to get a convergence estimate in total variation 
distance. Classically, the Brownian motion creates regularity and density. 
There are a lot of results giving regularity, in terms of initial points, of semigroup 
in small time. Let us quote the following result of Wang, which holds for processes
living on a manifold.

\begin{lem}[\cite{wang10}]\label{lem:lipschitz}
Suppose that $\sigma$ is constant and denote by $\eta$ the infimum
of its spectrum.  If $A$ is a $\cC^2$ function such that 
\[
\frac{1}{2}(\operatorname{Jac} A+\operatorname{Jac} A^T)
\geq K \operatorname{I}_d
\]
then there exists $K_\eta$ such that, for small $\varepsilon>0$, 
\[
\TVNRM{\delta_x P_\varepsilon-\delta_y P_\varepsilon}
\leq K_\eta\frac{\ABS{x-y}}{ \sqrt{\varepsilon}}.
\]
\end{lem}

\begin{rem}$\ $
\begin{itemize}
\item
There are many proofs leading to this kind of results, see for example 
Aronson \cite{aronson} for pioneering works, and \cite{wang10} using 
Harnack's  and Pinsker's inequalities.
\item Note that in \cite{GW}, an equivalent bound was given for the 
kinetic Fokker-Planck equation but with $\varepsilon$ replaced 
by $\varepsilon^3$.
\end{itemize}
\end{rem}

Now that we have a decay in Wasserstein distance and a control on the total variation distance after
a small time, we can use the same idea as for the TCP process. As a consequence, we get the following result.
\begin{thm} Assume that $\sigma$ is constant. 
Under Points 1. or 2. of Proposition \ref{prop:Wasserstein-diff}, 
one has, for any $\nu$ and $\tilde \nu$ in $\cP_1(\dR^d)$, 
\[
\TVNRM{\nu P_t-\tilde \nu P_t}\leq 
\frac{K e^{\lambda \varepsilon}}{\sqrt\varepsilon} 
W_1(\nu,\tilde \nu) e^{-\lambda t}.
\]
Under Point 3. of Proposition \ref{prop:Wasserstein-diff}, one 
has, for any $\nu$ and $\tilde \nu$ in $\cP_2(\dR^d)$, 
\[
\TVNRM{\nu P_t-\tilde \nu P_t}\leq 
\frac{K e^{\lambda \varepsilon}}{ \sqrt\varepsilon} 
W_2(\nu,\tilde \nu) e^{-\lambda t}.
\]
\end{thm}

\begin{proof}
 
 Using first Lemma ~\ref{lem:lipschitz} and then Point 1. of Proposition \ref{prop:Wasserstein-diff}, we get
 
 \begin{eqnarray*}
\TVNRM{\nu P_t-\tilde \nu P_t}&=&\TVNRM{\nu P_{t-\varepsilon}P_\varepsilon-\tilde \nu P_{t-\varepsilon}P_\varepsilon}\\
&\le&\frac K{\sqrt{\varepsilon}} W_1(\nu P_{t-\varepsilon},\tilde\nu P_{t-\varepsilon})\\
&\le&\frac{K e^{\lambda \varepsilon}}{\sqrt\varepsilon} 
W_1(\nu,\tilde \nu) e^{-\lambda t}.
 \end{eqnarray*}
 The proof of the second assertion is similar, except the use of Point 3. of Proposition \ref{prop:Wasserstein-diff} in the second step.
\end{proof}

\begin{rem}
Once again, one can give in the case of kinetic Fokker-Planck equation 
estimate in total variation distance, using the previous remarks 
(see \cite{MR2381160} for qualitative results).
\end{rem}

\bigskip
\noindent
\textbf{Acknowledgements.} Part of this work was done during a short (resp. long) visit of F.M. (resp. J.B.B.) to CMM (Santiago, Chile). They both thank colleagues from CMM for their warm hospitality, and particularly Joaquin Fontbona for interesting discussions. 
Another part of this work was done during a visit of F.M. and P.-A. Z. to the University of Neuch\^{a}tel; they 
thank Michel Bena\"im for his hospitality and additional interesting discussions. A.C. acknowledges the hospitality of IRMAR (Rennes, France). The research of A.C. is sponsored by MECESUP Doctoral Fellowship UCH 0410/07-10, MECESUP Doctoral Fellowship for sojourn UCH0607 and CONICYT Doctoral Fellowship/11. This research was partly supported by the French ANR projects EVOL and ProbaGeo.\\
We deeply thank the referee for his/her quick and constructive report.

\bibliographystyle{amsalpha}
\bibliography{tcptv}

\bigskip


{\footnotesize %
\noindent Jean-Baptiste~\textsc{Bardet},
e-mail: \texttt{jean-baptiste.bardet(AT)univ-rouen.fr}

\medskip

\noindent \textsc{UMR 6085 CNRS Laboratoire de Math\'ematiques 
Rapha\"el Salem (LMRS)\\ Universit\'e de Rouen,
Avenue de l'Universit\'e, BP 12,
F-76801 Saint Etienne du Rouvray, France}
  
\bigskip

\noindent Alejandra~\textsc{Christen}, 
e-mail: \texttt{achristen(AT)dim.uchile.cl}

\medskip 

\noindent\textsc{UMI 2807 UCHILE-CNRS DIM-CMM\\ Universidad de Chile, Casilla 170-3, Correo 3, Santiago, Chile\\
and\\
Instituto de Estadistica, Pontificia Universidad Cat\'olica de Valpara\'\i so}

\bigskip

\noindent Arnaud~\textsc{Guillin}, 
e-mail: \texttt{guillin(AT)math.univ-bpclermont.fr}

\medskip 

\noindent\textsc{UMR 6620 CNRS Laboratoire de Math\'ematiques,\\ Universit\'e Blaise Pascal, Campus des Cezeaux, 63177 Aubi\`ere cedex, France.\\
Institut Universitaire de France.}

\bigskip

 \noindent Florent~\textsc{Malrieu},
 e-mail: \texttt{florent.malrieu(AT)univ-rennes1.fr}

 \medskip

  \noindent\textsc{UMR 6625 CNRS Institut de Recherche Math\'ematique de
    Rennes (IRMAR) \\ Universit\'e de Rennes 1, Campus de Beaulieu, F-35042
    Rennes \textsc{Cedex}, France.}
\bigskip

\noindent Pierre-Andr\'e~\textsc{Zitt}, 
e-mail: \texttt{Pierre-Andre.Zitt(AT)u-bourgogne.fr}

\medskip 

\noindent\textsc{UMR 5584 CNRS Institut de Math\'ematiques de Bourgogne,\\
Universit\'e de Bourgogne, UFR Sciences et Techniques,\\
9 avenue Alain Savary -- BP 47870, 
21078 Dijon Cedex, France
}

}

\end{document}